\documentclass{article}
\usepackage{amssymb,amsmath,amsfonts,amsthm,amsxtra}
\usepackage{enumerate}
\usepackage{bbm}
\usepackage{accents}
\usepackage[utf8]{inputenc}

\makeatletter
\@namedef{subjclassname@2010}{%
  \textup{2010} AMS Mathematics Subject Classification}
\makeatother

\newtheorem{thm}{Theorem}[section]
\newtheorem{lemma}[thm]{Lemma}

\theoremstyle{definition}

\newtheorem{defn}[thm]{Definition}

\newtheorem{prob}[thm]{Problem}

\theoremstyle{remark}

\numberwithin{equation}{section}

\newcommand\mb[1]{\mathbb{#1}}

\newcommand\mc[1]{\mathcal{#1}}
\newcommand\mf[1]{\mathfrak{#1}}

\newcommand{\ignore}[1]{}

\begin{document}

\title{Set-theoretic Problems Concerning Lindelöf Spaces}

\author{Franklin D. Tall\makebox[0cm][l]{$^1$}}
\footnotetext[1]{The author was supported in part by NSERC Grant A-7354.\vspace*{1pt}}


\date{\today}
\maketitle
\begin{abstract}
I survey problems concerning Lindelöf spaces which have partial
set-theoretic solutions.
\end{abstract}
\renewcommand{\thefootnote}{}
\footnote
{\parbox[1.8em]{\linewidth}{$(2010)$ Mathematics Subject Classification. Primary 54A35, 54D20; Secondary 54A25, 54D30.}\vspace*{3pt}}
\renewcommand{\thefootnote}{}
\footnote
{\parbox[1.8em]{\linewidth}{Key words and phrases: Lindelöf, productively Lindelöf, powerfully Lindelöf,
  Alster, $\sigma$-compact, Menger, Hurewicz, Rothberger, $D$-space, Michael
  space, points $G_\delta$, indestructible, topological games}}

\emph{Lindelöf} spaces, i.e. spaces in which every open cover has a
countable subcover, are a familiar class of topological spaces. There
is a significant number of (mainly classic) problems concerning
Lindelöf spaces which are unsolved, but have partial set-theoretic
solutions. For example, consistency is known but independence is not;
large cardinals suffice but are not known to be necessary, and so
forth. The purpose of this note --- which is an expanded version of a
talk given at the 2010 BEST conference --- is to survey such questions
in the hope that set theorists will find them  worthy of
attention. Indeed we strongly suspect that the difficulty of these
problems is more set-theoretic than topological. Not much topological
knowledge is needed to work on them. Undefined terms can be found in
\cite{Engelking1989}. In Sections 1--3 we shall assume all spaces are
$T_3$, except for a remark at the end of Section 2.

I shall start with a collection of problems I have been investigating
for the past couple of years. Several of these problems are classic
and well-known; other are more specialized or recent, but are related
to the classic ones.

\section{Productively Lindelöf spaces}

\begin{defn}
A space $X$ is \emph{productively Lindelöf} if $X \times Y$ is Lindelöf,
for every Lindelöf $Y$. $X$ is \emph{powerfully Lindelöf} if $X^\omega$
is Lindelöf.
\end{defn}

Both of these concepts have been studied for a long time, but the
terminology is recent: \cite{Barr2007a}, \cite{AurichiTall}
respectively. Ernie Michael wondered more than thirty years ago whether:

\begin{prob}\label{prob1}
Is every productively Lindelöf space powerfully Lindelöf?
\end{prob}

This problem is raised in \cite{Przymusinski1984}. The problem is
presumably inspired by the fact that $\sigma$-compact spaces are both
productively Lindelöf and powerfully Lindelöf. Until recently, the
only significant result concerning Problem \ref{prob1} since
Przymusi\'{n}ski's survey was due to Alster \cite{Alster1988}, who solved it affirmatively under CH for regular spaces of weight $\leq \aleph_1$.
%

In the process of proving his \ignore{weaker version of Theorem \ref{thm1}}result, Alster introduced the
following concept, called \emph{Alster} in \cite{Barr2007a}.

\begin{defn}
A space $X$ is \emph{Alster} if whenever $\mc{G}$ is a cover of $X$ by
$G_\delta$'s such that each compact subset of $X$ is covered by finitely
many members of $\mc{G}$, then $\mc{G}$ has a countable subcover.
\end{defn}

\emph{Alster} fits in between $\sigma$-compact and productively Lindelöf:

\begin{thm}[\cite{Alster1988}]
  Every $\sigma$-compact space is Alster. Every Alster space is both
  productively Lindelöf and powerfully Lindelöf. $CH$ implies every
  productively Lindelöf space of weight $\le \aleph_1$ is Alster.
\end{thm}

\begin{prob}
  Is every productively Lindelöf space Alster?
\end{prob}
Nothing is known if CH fails.  Alster spaces in which compact sets are $G_\delta$ --- in particular,
metrizable spaces --- are easily seen to be $\sigma$-compact, while it
is also easy to see that $\sigma$-compact spaces are Alster.

We now turn to an apparently unrelated problem.

\section{$D$-spaces}

\begin{defn}
  A space $\langle X, \mc{T} \rangle$ is $D$ if whenever $f : X \to
  \mc{T}$ is an \emph{open neighborhood assignment}, i.e. $x \in
  f(x)$ for each $x \in X$, there is a closed discrete $D \subseteq X$ such that $\bigcup
  \{f(x) : x \in D \} = X$.
\end{defn}

There are two recent surveys of $D$-spaces: \cite{Eisworth2007} and
\cite{Gruenhage2009}. The major question of interest (implicit in \cite{vanDouwenPfeffer}) is:

\begin{prob}
  Is every Lindelöf space a $D$-space?
\end{prob}

Despite much effort, very little is known about this problem. A
breakthrough occurred when Aurichi \cite{Aurichi} showed that an
apparently minor strengthening of the Lindelöf property did yield $D$.

\begin{defn}
  A space $X$ is \emph{Menger} if whenever $\{ \mc{U}_n \}_{n <
    \omega}$ are open covers of $X$, there exist finite $\mc{V}_n
  \subseteq \mc{U}_n$ for each $n$, such that $\bigcup_{n<\omega} \mc{V}_n$ covers
  $X$.
\end{defn}

\begin{thm}[\cite{Aurichi}]\label{thm4}
  Menger spaces are $D$.
\end{thm}

In \cite{AurichiTall}, a connection between ``productively Lindelöf''
and ``$D$'' was established:

\begin{thm}
  Alster implies Menger.
\end{thm}

It is thus natural to ask:

\begin{prob}
  Does productively Lindelöf imply $D$?
\end{prob}

After a series of partial results, in \cite{TallMay} I proved:

\begin{thm}\label{thm6}
  $CH$ implies productively Lindelöf spaces are Menger and hence $D$.
\end{thm}
The proof of \ref{thm6} is easy and instructive.  Of course, ``productively Lindelöf'' is a strong hypothesis, but even
strengthenings of Lindelöf such as ``hereditarily Lindelöf'' are not
known --- even consistently --- to imply $D$. There is one more
set-theoretic result worth pointing out here:

\begin{thm}[\cite{AurJunLar}]
  It is consistent that every Lindelöf space of size $\aleph_1$ is $D$.
\end{thm}

This is now best seen as a consequence of Theorem \ref{thm4} together
with the following folklore result (see e.g. \cite{Tall2009a}):

\begin{thm}
  Lindelöf spaces of cardinality $< \mf{d}$ are Menger.
\end{thm}

The key to proving Theorem \ref{thm6} lies in a concept introduced by
Arhangel'ski\u\i{} \cite{Arhangelskii2000}:

\begin{defn}
  A space $X$ is \emph{projectively $\sigma$-compact} if every
  continuous image of $X$ in a separable metrizable space is
  $\sigma$-compact.
\end{defn}

One can make analogous definitions, e.g. of \emph{projectively
  Menger}. There are projectively $\sigma$-compact space that are not
$\sigma$-compact --- see \cite{Arhangelskii2000} and
\cite{TallMay}. Arhangel'ski\u\i{} in this paper also in effect
proved:

\begin{thm}\label{thm9}
  Lindelöf projectively Menger spaces are Menger.
\end{thm}

This was specifically stated and proved in \cite{Kocinac} and
\cite{BonCamMat}. The method of proof is worth pointing out, since it
often allows one to take results about Lindelöf metrizable spaces and
transfer them to arbitrary Lindelöf spaces.

\begin{lemma}[\cite{Engelking1989}]\label{lemma10}
  Let $\mc{U}$ be an open cover of a Lindelöf space $X$. There is then
  a continuous function $f : X \to Y$, $Y$ separable metrizable, and
  an open cover $\mc{V}$ of $Y$ such that $\{ f^{-1}(V) : V \in
  \mc{V}\}$ refines $\mc{U}$.
\end{lemma}

To prove Theorem \ref{thm9}, given a sequence $\{ \mathcal{U}_n \}_{n
  < \omega}$ of such covers, find the corresponding $f_n$'s, $Y_n$'s
and $\mathcal{V}_n$'s. Then the diagonal product of the $f_n$'s maps
$X$ onto a subspace $\hat{Y}$ of $\prod Y_n$. $\hat{Y}$ is
$\sigma$-compact, hence Menger, so we can take finite subsets of the
$\mathcal{V}_n$'s forming a cover and then pull them back to $X$ to
find the required finite subsets of the $\mathcal{U}_n$'s.

The other ingredient in the proof of Theorem \ref{thm9} is:

\begin{thm}\label{thm11}
  $CH$ implies productively Lindelöf spaces are projectively
  $\sigma$-compact.
\end{thm}

This was in effect proved in \cite{Michael1971}, but specifically
stated and proved in \cite{Alster1987}. The method of proof is again
exemplary:

\begin{proof}
  Let $X$ be a separable metrizable image of a productively Lindelöf
  $P$. Then $X$ is productively Lindelöf. (Separable metrizable is not
  needed for this.)
  Embed $X$ in $[0,1]^{\aleph_0}$. $[0,1]^{\aleph_0}$ has a countable
  base, so by $CH$ we can take open subsets $\{ U_\alpha \}_{\alpha <
    \omega_1}$ of $[0,1]^{\aleph_0}$ such that every open set about $Y
  = [0,1]^{\aleph_0} - X$ includes some $U_\alpha$. By taking
  countable intersections, we can find a decreasing sequence
  $\{G_\beta\}_{\beta < \omega_1}$ of $G_\delta$'s about $Y$, such
  that every open set about $Y$ includes some $G_\beta$. If $X$ is not
  $\sigma$-compact, we can assume the $G_\beta$'s are strictly
  descending. Pick $p_\beta \in (G_{\beta+1} - G_\beta) \cap X$. Put a
  topology on $Z = Y \cup \{ p_\beta : \beta < \omega_1\}$ by
  strengthening the subspace topology to make each $\{ p_\beta \}$
  open.  Then $Z$ is Lindelöf, but $X \times Z$ is not, since $\{
  \langle p_\beta, p_\beta \rangle : \beta < \omega_1 \}$ is closed
  discrete.
\end{proof}

Although $T_3$ (and hence paracompactness) is a natural assumption in the context of wondering whether Lindel\"{o}f spaces are $D$-spaces, given the lack of counterexamples even the following very recent result is noteworthy:

\begin{thm}[\cite{SZ}]
$V=L$ implies there is a hereditarily Lindel\"{o}f $T_1$ space of size $\aleph_1$ which is not $D$.
\end{thm}

\section{Michael spaces}

In \cite{Michael1963}, Ernie Michael proved:

\begin{thm}\label{thm12}
  $CH$ implies there is a Lindelöf space $X$ such that $X \times
  \mb{P}$ (the space of irrationals) is not Lindelöf.
\end{thm}

The question of whether such a space can be constructed without
additional axioms has become known as \emph{Michael's problem}, and
such a space is called a \emph{Michael space}. Notice that the proof
of Theorem \ref{thm11} is also a proof of Theorem \ref{thm12}:
$\mb{P}$ is not $\sigma$-compact, so under $CH$, it cannot be
productively Lindelöf.

\begin{prob}
  Is there a Michael space?
\end{prob}

Michael spaces have been constructed under a variety of assumptions
about small cardinals --- see e.g. \cite{Alster1990},
\cite{Lawrence1990} and \cite{Moore1999}. These yield:

\begin{thm}
  $\mf{b} = \aleph_1$ or $\mf{d} = cov(\mc{M})$ imply there is a
  Michael space. ($\mc{M}$ is the meagre ideal for $\mb{P}$.)
\end{thm}

A unified approach to the various constructions is given in \cite{AAJT}.
There is a close connection between Michael's problem and

\begin{prob}\label{prob6}
  Is every productively Lindelöf metrizable space $\sigma$-compact?
\end{prob}

We have already seen that the answer is positive under $CH$ (Theorem
\ref{thm11}), but if we add on definability conditions for the space,
we can reduce that hypothesis to the existence of a Michael space:

\begin{thm}[\cite{Tall2009a}]
  Every productively Lindelöf analytic metrizable space is
  $\sigma$-compact if and only if there is a Michael space. Assuming
  the axiom of Projective Determinacy ($PD$) \emph{analytic} can be
  improved to \emph{projective}.
\end{thm}

  The key is to apply the \emph{Hurewicz Dichotomy}.

  \begin{defn}
    A class $\mc{K}$ of space satisfies the \emph{Hurewicz Dichotomy} if
    every member of $\mc{K}$ is either $\sigma$-compact or includes a
    closed copy of $\mb{P}$.
  \end{defn}

Analytic Lindel\"{o}f metrizable spaces satisfy the Hurewicz Dichotomy
\cite{Hurewicz1925} (or see \cite{Kechris1995}); projective Lindel\"{o}f metrizable
spaces do under $PD$ \cite{KLW1987} --- see discussion in
\cite{Tall2009a}. How far the Hurewicz Dichotomy can be extended,
either in the metrizable or non-metrizable context is unclear. See
\cite{TallNote} and \cite{TallMay} for discussion in terms of
Problem \ref{prob6} and

\begin{prob}
  Is every productively Lindelöf space Menger?
\end{prob}

Notice that by Lemma \ref{lemma10} this is equivalent to:

\begin{prob}
  Is every productively Lindelöf metrizable space Menger?
\end{prob}

We should mention the well-known facts that $\sigma$-compact spaces
are Menger, but not necessarily the converse \cite{ChaberPol}, \cite{Ts}, and that $\mb{P}$ is not
Menger.

After this note was first submitted, D. Repov\v{s} and L. Zdomskyy \cite{RZ} succeeded in completely removing the definability condition, but at a cost, obtaining:

\begin{thm}
There is a Michael space if and only if productively Lindel\"{o}f metrizable spaces are Menger.
\end{thm}

The question about Hurewicz dichotomy is still of interest, however, in terms of obtaining $\sigma$-compactness.  In addition to \cite{TallProblems}, see \cite{RZ}, wherein the following problem is raised:
\begin{prob}
Suppose there is a Michael space.  Is every co-analytic productively Lindel\"{o}f metrizable space $\sigma$-compact?
\end{prob}

\begin{defn}
A $T_3$ Lindel\"{o}f space is \emph{Hurewicz} if whenever $Z$ is a \v{C}ech-complete space including $X$, there is a $\sigma$-compact $Y$ such that $X \subseteq Y \subseteq Z$.
\end{defn}

In \cite{AurichiTall} we proved:

\begin{thm}
$\mathfrak{d} = \aleph_1$ implies productively Lindel\"{o}f metrizable spaces are Hurewicz.
\end{thm}

We thus have that the three progressively weaker hypotheses: \textit{CH}, $\mathfrak{d} = \aleph_1$, \textit{there is a Michael space}, imply the respectively weaker conclusions: \textit{$\sigma$-compact}, \textit{Hurewicz}, \textit{Menger}, for productively Lindel\"{o}f metrizable spaces.

\begin{prob}
Are the stronger hypotheses necessary in order to obtain the stronger conclusions?
\end{prob}

\section{The cardinality of Lindelöf spaces with points $G_\delta$}

A. V. Arhangel'ski\u\i{} \cite{Arhangelskii1969} solved a famous problem
of P. S. Alexandroff by proving that:

\begin{lemma}
  Lindelöf first countable $T_2$ spaces have cardinality $\le 2^{\aleph_0}$
\end{lemma}

A proof using elementary submodels can be found in chapter 24 of
\cite{Just1997}. Since in $T_1$ first countable spaces points are
$G_\delta$, Arhangel'ski\u\i{} then asked whether the continuum also
bounded the cardinality of Lindelöf $T_2$ spaces with points
$G_\delta$. The consistency of a positive answer remains open. I have
surveyed this question twice before: \cite{Tall1995}, \cite{TallProblems},
so after a minimum amount of background I will turn to new material. A
consistent negative answer was provided by \cite{Shelah1996}; an
easier and more general proof was provided by \cite{Gorelic1993}:

\begin{thm}
  It is consistent with $ZFC+CH$ that there is a Lindelöf
  zero-dimensional space with points $G_\delta$ of size
  $2^{\aleph_1}$.
\end{thm}

A key concept concerning  this problem is given in the following:

\begin{defn}
  A Lindelöf space is \emph{indestructible} if it remains Lindelöf
  after countably closed forcing.
\end{defn}

There is a combinatorial characterization of indestructibility we
shall mention in the next section, on topological games. There is also
interest in the general question of what forcings preserve
Lindelöfness --- see \cite{Kada}. It is known that adding at least
$\aleph_1$ Cohen reals makes a Lindelöf space indestructible
\cite{Dow1988} but the following problem is open:

\begin{prob}
  After adding one Cohen real, does a Lindelöf space become
  indestructible?
\end{prob}

The reason why indestructibility is of interest is because:

\begin{thm}[\cite{Tall1995}]\label{thm4p5}
  Collapse a supercompact cardinal to $\omega_2$ by countably closed
  forcing. Then every indestructible Lindelöf space with points
  $G_\delta$ has cardinality $\le 2^{\aleph_0}$.
\end{thm}

We thus want to know:

\begin{prob}
  Does $CH$ imply every Lindelöf $T_2$ space with points $G_\delta$ is
  indestructible?
\end{prob}

Although the method of proof of Theorem \ref{thm4p5} is a good illustration of supercompact reflection, Theorem \ref{thm4p5} itself has been significantly improved by Marion Scheepers to obtain:
\begin{thm}[\cite{Scheepers2009}]
Collapse a measurable cardinal to $\omega_2$ by countably closed forcing.  Then every indestructible Lindel\"{o}f space with points $G_\delta$ has cardinality $\leq 2^{\aleph_0}$.
\end{thm}
Scheepers uses games on an ideal.

In \cite{ScheepersTall} Scheepers and the author obtained positive
consistency results by strengthening Lindelöf to \emph{Rothberger}.  Their use of supercompact reflection was again improved to measurable reflection in \cite{Scheepers2009}.

\begin{defn}
  A space is \emph{Rothberger} if whenever $\{ \mc{U}_n
  \}_{n<\omega}$ are open covers of it, there are $U_n \in \mc{U}_n$,
  each $n \in \omega$, such that $\{U_n\}_{n<\omega}$ is a cover.
\end{defn}

\begin{thm}[\cite{Scheepers2009}]
  Collapse a measurable to $\omega_2$ by countably closed forcing. Then Rothberger spaces with points $G_\delta$ have size
  $\le \aleph_1$.
\end{thm}

This is because Rothberger space are indestructible, which is fine,
but ``Rothberger'' is too strong --- for sets of reals it implies
strong measure zero. The Menger property we mentioned before is a
weakening of Rothberger; we therefore ask:

\begin{prob}
  Is it consistent with $ZFC$ that every Menger $T_2$ space with
  points $G_\delta$ has cardinality $\le 2^{\aleph_0}$?
\end{prob}

This is reasonable since ``Menger'' is also a weakening of
``$\sigma$-compact'' and we have:

\begin{thm}
  Every $\sigma$-compact $T_2$ space with points $G_\delta$ has
  cardinality $\le 2^{\aleph_0}$.
\end{thm}

\begin{proof}
  Compact $T_2$ spaces with points $G_\delta$ are first countable.
\end{proof}

Arhangel'ski\u\i{} also asked:

\begin{prob}
  Do first countable $T_1$ Lindelöf spaces have cardinality $\le
  2^{\aleph_0}$?
\end{prob}

Points are $G_\delta$'s in such a space, so the usual positive
consistency results apply, but so far we have not been able to go
beyond these. No counterexamples are known, even consistently. A
natural try is to take a Lindelöf space with points $G_\delta$ and
consider the weaker topology generated by witnesses for ``$G_\delta$''
about each point. Unfortunately, there is no reason to believe the
resulting topology is first countable. If we strengthen ``points
$G_\delta$'' to ``point-countable $T_1$-separating open cover'',
i.e. an open cover $\mc{U}$ such that each point is in only countably
many members of the cover and for each $x \neq y \in X$, there is a $U
\in \mc{U}$ such that $x \in U$ and $y \notin U$, then when we take
$\mc{U}$ as a subbase for a topology, that topology is first
countable, so we would have that if there were a Lindelöf space of
size $> 2^{\aleph_0}$ with a point-countable $T_1$-separating open
cover, then there would be a Lindelöf first countable $T_1$ space of
size $> 2^{\aleph_0}$. However, in fact Lindelöf spaces with
point-countable $T_1$-separating open covers have cardinality $\le
2^{\aleph_0}$ \cite{Charlesworth1977} so this is a dead end.

\section{Topological Games}

\begin{defn}
  The \emph{Menger game} is an $\omega$-length perfect information
  game in which in the $n$th inning, ONE chooses an open cover
  $\mc{U}_n$ and TWO chooses a finite $\mc{V}_n \subseteq
  \mc{U}_n$. TWO wins if $\{\bigcup \mc{V}_n : n < \omega \}$ covers
  $X$. Otherwise, ONE wins.
\end{defn}

\begin{thm}[\cite{Hurewicz1925}]
  $X$ is Menger if and only if ONE has no winning strategy in the
  Menger game on $X$.
\end{thm}

Telgársky \cite{Telgarsky} proved:

\begin{thm}\label{thm16}
  For metrizable $X$, $X$ is $\sigma$-compact if and only if TWO has a
  winning strategy in the Menger game on $X$.
\end{thm}

Scheepers \cite{Scheepers1995} provided a much more accessible proof.

In an earlier version of this survey and in \cite{TallMay} we asked whether ``metrizable" was necessary in Theorem \ref{thm16}.  T. Banakh and L. Zdomskyy \cite{Z} proved that it wasn't, replacing ``metrizable" by ``hereditarily Lindel\"{o}f $T_3$".

\ignore{
\begin{prob}\cite{TallMay}\label{prob9}
  For Lindelöf $X$, is $X$ projectively $\sigma$-compact if and only
  if TWO has a winning strategy in the Menger game on $X$?
\end{prob}
}
Several variations on the Menger property and the Menger game have
been considered. See e.g. \cite{JMSS} and \cite{ScheepersTall}. One
can vary the kind of cover, the kind of subsets of the cover, and the
length of the game. Notice that if ONE 
does not have a winning strategy in the Menger game, then clearly the
space is Menger. The converse is true, but not so obvious. That remark
in general applies to $\omega$-length games of this sort. See
e.g. \cite{Hurewicz1925} and \cite{Pawlikowski1994}. However, for
$\omega_1$-length games, although again ONE 
having no winning strategy easily implies the $\omega_1$-version of
Menger, the converse is not clear. The notation used by Scheepers is
convenient: $G_{\mathrm{FIN}}^{\omega}(\mc{O}, \mc{O})$ is the Menger
game; $G_{1}^{\omega}(\mc{O},\mc{O})$ is the \emph{Rothberger game}
in which we only pick one element from each
cover. $G_{\mathrm{FIN}}^{\omega_1}(\mc{O},\mc{O})$,
$G_{1}^{\omega_1}(\mc{O},\mc{O})$ are respectively the
$\omega_1$-versions of the Menger and Rothberger games. The selection
principles $S_{\mathrm{FIN}}^{\omega_1}(\mc{O},\mc{O})$
($S_{1}^{\omega_1}(\mc{O},\mc{O})$) assert that it is possible to pick
finitely many elements (one element) from each cover in a
$\omega_1$-sequence and get a cover. The problem of determining
whether $S_{1}^{\omega_1}(\mc{O},\mc{O})$ implies ONE has no winning
strategy in $G_{1}^{\omega_1}(\mc{O},\mc{O})$ is particularly
interesting and is asked in \cite{ScheepersTall}:

\begin{prob}
  Let $X$ be a Lindelöf space. Suppose $X$ satisfies
  $S_{1}^{\omega_1}(\mc{O},\mc{O})$
  ($S_{\mathrm{FIN}}^{\omega_1}(\mc{O},\mc{O})$). Does ONE not have a winning
  strategy in $G_{1}^{\omega_1}(\mc{O},\mc{O})$
  ($G_{\mathrm{FIN}}^{\omega_1}(\mc{O},\mc{O})$) on $X$?
\end{prob}

It is interesting because of:

\begin{thm}[\cite{ScheepersTall}]
  ONE has no winning strategy in $G_{1}^{\omega_1}(\mc{O},\mc{O})$ for
  a Lindelöf space $X$ if and only if $X$ is indestructible.
\end{thm}

\section{Lindelöf spaces with no small Lindelöf subspaces}

Hajnal and Juhász \cite{Hajnal1976} asked:

\begin{prob}
  Does every Lindelöf space of size $\aleph_2$ have a Lindelöf
  subspace of size $\aleph_1$?
\end{prob}

An extensive discussion of what is known and why the problem is
interesting can be found in \cite{Baumgartner2002}. There is a
consistent counterexample:

\begin{thm}[\cite{Koszmider2002}]
  It is consistent with $GCH$ that there is a Lindelöf space of size
  $\aleph_2$ with no Lindelöf subspaces of size $\aleph_1$.
\end{thm}

The example is constructed by some moderately difficult countably
closed forcing. It is a \emph{$P$-space}, i.e. $G_\delta$'s are
open. The key observation is that for Lindelöf $P$-spaces, having a
Lindelöf subspace of size $\aleph_1$ is equivalent to having a
convergent $\omega_1$-sequence. For $P$-spaces, Kozmider and the
author also proved:

\begin{thm}[\cite{Koszmider2002}]
  It is consistent that $CH$, $2^{\aleph_1} > \aleph_2$, and every
  Lindelöf $P$-space of size $\aleph_2$ has a Lindelöf subspace of
  size $\aleph_1$.
\end{thm}

\begin{prob}
  Is it consistent with $GCH$ that every Lindelöf $P$-space of size
  $\aleph_2$ has a Lindelöf subspace of size $\aleph_1$?
\end{prob}

As for positive results, the obvious approach is to try a reflection
argument. If one adds additional topological hypotheses, one can use
elementary submodels to obtain some results, see
\cite{Baumgartner2002}. A more vigorous attack can be mounted using
large cardinals. As usual (see \cite{TallProblems} for an exposition), one
wants to show $j"X$ is a Lindelöf subspace of $j(X)$, so one needs to
worry about preservation of Lindelöfness under forcing, and whether
$j"X$ is a subspace of $j(X)$, where $j$ is a generic embedding
obtained from forcing a large cardinal to collapse. One then has:

\begin{thm}[\cite{ScheepersTall}]
  If it's consistent that there is a supercompact cardinal, it's
  consistent with $GCH$ that every Rothberger space of size $\aleph_2$
  and character $\le \aleph_1$ has a Rothberger subspace of size
  $\aleph_1$.
\end{thm}

The details can also be found in \cite{ScheepersTall} or in my survey
article \cite{TallProblems}. The key point is that Rothberger spaces are
\emph{indestructibly Rothberger}, i.e. they remain Rothberger after
countably closed forcing \cite{ScheepersTall}. Scheepers
\cite{Scheepers2009} later improved this via an ideal game, so as to
only use a measurable cardinal.

If one only assumes the given space is Lindelöf, consistency results
are hard to come by. In \cite{Baumgartner2002} the following result is
obtained:

\begin{thm}\label{thm6.6}
  Assume there is a huge cardinal. Then it is consistent with $GCH$
  that every Lindelöf first countable space of size $\aleph_2$ has a
  Lindelöf subspace of size $\aleph_1$.
\end{thm}

The proof uses the Kunen-collapse of the huge cardinal $\kappa$ to
$\aleph_1$ and its target $j(\kappa)$ to $\aleph_2$
\cite{Kunen1976}. Since Lindelöf first countable $T_2$ spaces have
cardinality $\le 2^{\aleph_0}$, this is not so interesting.

The problem of getting $j"X$ to be a subspace of $j(X)$ leads me to
propose the following problem:

\begin{prob}
  If $X$ is Lindelöf of size $\aleph_2$, does there exist a subspace
  of $X$ of size $\aleph_1$ with a weaker Lindelöf topology?
\end{prob}

\begin{thm}
  If it is consistent that there is a huge cardinal, then it is
  consistent with $GCH$ that if $X$ is Lindelöf of size $\aleph_2$,
  then there is a subspace of $X$ of size $\aleph_1$ with a weaker
  Lindelöf topology.
\end{thm}

This is a corollary of the proof of \ref{thm6.6}. First countability is
only used to get $j"X$ to be a subspace of $j(X)$.

I do not know whether the \cite{Koszmider2002} example is a negative
solution to this problem.

In \cite{Hajnal1976}, Hajnal and Juhász prove:

\begin{thm}
  $CH$ implies every uncountable compact $T_2$ space has a Lindelöf
  subspace of size $\aleph_1$.
\end{thm}

\begin{defn}
  A space $X$ is \emph{countably tight} if whenever $x \in \overline{Y}$, $Y \subseteq X$, there is a countable $Z \subseteq Y$
  such that $x \in \overline{Z}$.
\end{defn}

Surprisingly, for spaces which are not countably tight, they do not
need $CH$. Thus we have:

\begin{prob}
  Does $ZFC$ imply every uncountable countably tight compact space has a
  Lindelof subspace of size $\aleph_1$?
\end{prob}

I. Juhász has reminded me that even the first countable case is open.

\section{Products of Lindelöf spaces}

One of the most important results in General Topology is the Tychonoff
Theorem: \emph{the product of any family of compact spaces is
  compact}.  It is natural to attempt to generalize this to Lindelöf
spaces, but the Sorgenfrey Line (right half-open interval topology on
$\mb{R}$) provides an easy counterexample. It is not clear whether
there is any reasonable upper bound on the Lindelöf number of a
product of Lindelöf spaces even in simple cases. For concreteness we
ask:

\begin{prob}
  Is $L(X \times Y) \le 2^{\aleph_0}$, where $X$ and $Y$ are Lindelöf,
  and $L(Z)$ is the least cardinal such that every open cover of $Z$
  has a subcover of size $\le \kappa$?
\end{prob}

Surprsingly, the only known bound for the Lindelöf number of $X \times
Y$, for $X$, $Y$ Lindelöf, is the first strongly compact cardinal
\cite{Juhasz1984}! Shelah \cite{Shelah1996} forced a variation of his
large Lindelöf space with points $G_\delta$ to produce Lindelöf spaces
$X$, $Y$ with the Lindelöf number of $X \times Y > 2^{\aleph_0}$. This
was simplified by \cite{Gorelic1994}. There has been no progress
since. The strongly compact bound is obtained by mindlessly
generalizing the proof of the Tychonoff Theorem. Surely one ought to
be able to do better. Perhaps one can obtain some results by
strengthening the Lindelöf property to, e.g.  Rothberger.

\section{The Lindelöf number of the $G_\delta$ topology on a Lindelöf space}

\begin{prob}
Can the Lindelöf number of the topology generated by making $G_\delta$'s
open exceed $2^{\aleph_0}$ if the original topology is Lindelöf?
\end{prob}

The \cite{Gorelic1993} or \cite{Shelah1996} examples trivially show
that it can, but is there an example in ZFC?

\section{Conclusion}
I thank the referee for
several improvements.

One could add more problems to the list we have given here, but these will
keep the reader busy enough.

\nocite{*}
\bibliographystyle{amsalpha}
\bibliography{lindelof}
{\rm Franklin D. Tall, Department of Mathematics, University of
Toronto, Toronto, Ontario M5S 2E4, Canada}

{\noindent\it e-mail address:} {\rm tall@math.utoronto.ca}
\end{document}